\documentclass[twoside, 12pt]{amsart}
\usepackage[headings]{fullpage}
\usepackage{amssymb,amsmath,amsopn}

\newtheorem{thm}{Theorem}

\newtheorem{prop}{Proposition}
\theoremstyle{definition}
\newtheorem{defn}{Definition}
\newtheorem{lemma}{Lemma}

\renewcommand{\Re}{\mathbb R}
\newcommand{\Red}{\Re^d}

\newcommand{\F}{\mathcal{F}}
\newcommand{\CC}{\mathcal{C}}
\newcommand{\G}{\mathcal{G}}
\newcommand{\ck}[3]{c_{#1}^{#3}(#2)}
\newcommand{\ik}[3]{i_{#1}^{#3}(#2)}
\newcommand{\Pe}{\mathbb{P}}
\newcommand{\C}{\mathcal{C}}

\title{Separation with restricted families of sets}
\date{}
\author[
Zs. L\'angi \and M. Nasz\'odi \and J. Pach \and G. Tardos \and G. T\'oth
]{Zsolt L\'angi \and M\'arton Nasz\'odi \and J\'anos Pach \and G\'abor
Tardos \and G\'eza T\'oth}
\address[Zs. L\'angi]{Budapest University of Technology, 
Budapest}\email{zlangi@math.bme.hu}
\address[M. Nasz\'odi]{EPFL, Lausanne and E\"otv\"os University, 
Budapest}\email{marton.naszodi@math.elte.hu}
\address[J. Pach]{EPFL, Lausanne and R\'enyi Institute, 
Budapest}\email{pach@cims.nyu.edu}
\address[G. Tardos]{R\'enyi Institute, 
Budapest}\email{tardos.gabor@renyi.mta.hu}
\address[G. T\'oth]{R\'enyi Institute, Budapest and Budapest University of 
Technology}\email{toth.geza@renyi.mta.hu}

\thanks{
Zsolt L\'angi was supported by the Janos Bolyai Research Scholarship 
of the Hungarian Academy of Sciences.\\
M\'arton Nasz\'odi was supported by the Janos Bolyai Research Scholarship 
of the Hungarian Academy of Sciences and by the Hungarian National Science 
Foundation (OTKA) Grant PD-104744.\\
J\'anos Pach was supported by Swiss National Science Foundation 
Grants 200020-144531 and 200021-137574.\\
G\'abor Tardos was supported by the \emph{Cryptography} ``Lend\"ulet'' 
project of the Hungarian Academy of Sciences.\\
G\'eza T\'oth was supported by Hungarian National Science 
Foundation (OTKA) Grants K-111827 and K-83767.
}

\keywords{Search theory, separation, VC-dimension, Erd\H os--Szekeres theorem}
\subjclass[2010]{90B40, 52A37}

\begin{document}
\begin{abstract}
Given a finite $n$-element set $X$, a family of subsets ${\mathcal F}\subset
2^X$ is said to \emph{separate} $X$ if any two elements of $X$ are separated by
at least one member of $\mathcal F$. It is shown that if $|\mathcal
F|>2^{n-1}$, then one can select $\lceil\log n\rceil+1$ members of $\mathcal F$
that separate $X$. If $|\mathcal F|\ge \alpha 2^n$ for some $0<\alpha<1/2$,
then
$\log n+O(\log\frac1{\alpha}\log\log\frac1{\alpha})$ members of $\mathcal F$
are
always sufficient to separate all pairs of elements of $X$ that are separated
by some member of $\mathcal F$. This result is generalized to simultaneous
separation in several sets. Analogous questions on separation by families of
bounded Vapnik-Chervonenkis dimension and separation of point sets in
${\mathbb{R}}^d$ by convex sets are also considered.
\end{abstract}

\maketitle

\section{Introduction}

For a set $X$, we say that a subset of $X$ \emph{separates} two
elements if it contains one of them and does not contain the other. For a
family $\F$ of subsets of
$X$, we say that it \emph{separates} a pair of elements of $X$ if at least one
member
of $\F$ separates them. Furthermore, $\F$
\emph{separates} $X$ if every pair of distinct elements of $X$ is separated by
$\F$.

Suppose your computer is infected by a virus $x\in X$, where $X$ is the set of
known computer viruses. You want to perform a number of tests to find out
which virus it is. Each test detects a certain set of viruses, which can be
associated with the test. Let $\F$ denote the family of subsets of $X$
associated with the tests you can perform. These tests are sufficient to
identify the virus if, and only if $\F$ separates $X$. The number of tests
necessary is at least $\log|X|$, where $\log$ stands for the base 2 logarithm.
On the other hand, there is a family $\F\subset 2^X$ with
$|\F|\leq\lceil\log|X|\rceil$ that separates $X$. This is the starting point of
a rich discipline called \emph{combinatorial search theory}; see~\cite{AhlW87}.

Any fixed pair of distinct elements in $X$ is separated by $2^{|X|-1}$ subsets
of $X$, thus a family $\F$ with $|\F|>2^{|X|-1}$ separates $X$. Our first
theorem states that in this case, even a small subfamily of $\F$ does the job.

\begin{thm}\label{thm:corOfThmLogP1}
Let $X$ be a finite set, $\F \subseteq 2^{X}$ with $| \F | > 2^{|X|-1}$.
Then $X$ can be separated by a subfamily $\G\subseteq\F$ of cardinality at most
$\lceil \log |X| \rceil +1$.
\end{thm}

This statement is almost tight, but not completely. Indeed, for $|X|=5$,
Theorem~\ref{thm:corOfThmLogP1} guarantees the existence of a $4$-member
separating family, but it is easy to verify that $3$ sets suffice. In the
following generalization we give the best possible bound.

\begin{thm}\label{thm:logP1}
Let $X_1, \ldots, X_k$ be pairwise disjoint sets with $|X_i| \leq n$ for
$i=1,2,\ldots,k$. Let $X = \bigcup_{i=1}^k X_i$.
If $\F \subseteq 2^{X}$ satisfies $| \F | > 2^{|X|-1}$, then $\F$ has
a subfamily of cardinality at most $\lceil \log n \rceil +1$ that separates
$X_i$ for every $i\; (i=1,2,\ldots,k)$.

The above bound is tight, that is, the same statement is false for every
$n\ge 2$, if we replace $\lceil\log n\rceil + 1$ by $\lceil\log n\rceil$.
\end{thm}

We call $|\F|/2^{|X|}$ the \emph{density} of $\F$. If this slips below $1/2$,
we cannot guarantee the existence of a small subfamily separating $X$, as $\F$
itself
does not necessarily separate $X$. But, as claimed by the next theorem, we can
still find a small subfamily separating all pairs in $X$ that $\F$ separates.
We state this result for simultaneous separation.

\begin{thm}\label{thm:logPalpha}
Let $X_1, \ldots, X_k$ be disjoint sets with $|X_i| \leq n$ for all
$i=1,2,\ldots,k$. Let $X = \bigcup_{i=1}^k X_i$ and $\F \subseteq 2^{X}$. Then
$\F$ has a subfamily of size at most $\lceil \log n \rceil +C \log
\frac{1}{\alpha} \log \log \frac{1}{\alpha}$ separating every pair in each
$X_i$ that is separated by $\F$.
Here $\alpha=|\F|/2^{|X|}$ is the density of $\F$ and $C$ is a universal
constant.
\end{thm}


\begin{defn}\label{defn:constraint}
Let $X$ be a set of $n$ elements. We call a pair $(V,W)$ of disjoint subsets of
$X$ a \emph{constraint}, $V\cup W$ is the \emph{support} and $|V\cup W|$ is the
\emph{size} of the constraint.
A subset $A$ of $X$ \emph{satisfies} the constraint $(V,W)$ if $V\subseteq A$
and $W\cap A=\emptyset$. A family $\F$ of subsets of $X$ \emph{satisfies} a
constraint if the constraint is satisfied by some member of $\F$.
\end{defn}

Note that the fact that two elements $x,y\in X$ are separated by a member of
$\F$ means that at least one of the constraints $(\{x\},\{y\})$ and
$(\{y\},\{x\})$ is satisfied by $\F$.

Given a family $\F\subseteq2^X$ and a family of constraints satisfied by $\F$,
we are looking for a small subfamily of $\F$ that also satisfies the given
constraints. The next theorem establishes what the density of $\F$ has to be
(depending on the size of the constraints) for this to be possible.

\begin{thm}\label{thm:satcond}
For a positive integer $m$ and $1-\frac{1}{2^{m-1}}<\alpha<1$ there exists a
constant
$0<c=c(m,\alpha)$ with the following property.
If $X$ is a finite set, $\F$ is a family of subsets of $X$ with density
above $\alpha$ and $\CC$ is a collection of $N$ constraints, each of size
$m$ and satisfied by $\F$, then there is a subfamily consisting of at most
$c\log N$ sets from $\F$ that satisfies all the constraints in $\CC$.

A similar statement is false for any $m\ge1$ and $\alpha=1-\frac1{2^{m-1}}$.
\end{thm}

Up to this point, the ground set $X$ and the family $\F$ were not assumed to
possess any structure. The first such assumption we make is that of a bounded
VC-dimension and prove a linear lower bound on the size of any separating
subfamily in this case.
Given a family $\F$ of sets, the \emph{Vapnik-Chervonenkis dimension} (in
short, VC-dimension) of $\F$ is the largest integer $d$ for which there exists
a $d$-element set $A$ such that for every subset $B\subseteq A$, some member
$F$ of $\F$ has $F\cap A=B$.

Let us fix $d$
and assume a set of size $n$ is separated by a family $\F$ of VC-dimension
$d$. In this case $|\F|$ must be at least polynomial in $n$, namely $|\F|\ge
n^{1/(2^{d+1}-1)}$. This follows from the Shatter Function Lemma (or 
Sauer-Shelah
Lemma, cf. Lemma 10.2.5 in \cite{Mat02}) and the fact that the dual of
VC-dimension $d$ family has VC-dimension below $2^{d+1}$. The size of a
separator family of fixed VC-dimension $d$ can, indeed, be a small polynomial
of $n$, namely $|\F|<2^dn^{1/(2^d-1)}$ can be obtained by considering the dual
of the set system $\binom{Y}{2^d-1}$.

We can show even stronger, linear lower bound on the size of a separating
family if the base set to be separated can be arbitrarily chosen from an
infinite universe with a bounded VC-dimension set system.

\begin{thm}\label{thm:VC}
Let $U$ be an infinite set and ${\mathcal{F}}\subseteq2^U$ a family of subsets
of $U$ of Vapnik-Chervonenkis dimension $d$. For every $n>0$, there is a set
$X\subset U$ with $|X|=n$ such that any subfamily of ${\mathcal{F}}$ which
separates $X$, has at least $\frac{n-1}{d}$ members.
\end{thm}

Note that this theorem is almost tight. Let $U$ be the set of positive integers
and let $\F$ consist of the subsets of $U$ containing $1$ and separating at
most $d$ pairs $(i,i+1)$. The VC-dimension of $\F$ is $d$ and any $n$ element
subset of $U$ can be separated by $(n-1)/d+\log d$ members of $\F$.

Now, we turn to separation problems where $X$ and $\F$ have some geometric
structure.

\begin{defn}\label{def:genpos}
A set of points in $\Red$ is said to be \emph{in general position} if no $d+1$
points lie on a $(d-1)$-dimensional affine subspace.
\end{defn}

A prototype of separation questions in the geometric setting was first studied
by Gerbner and T\'oth \cite{GT13}. They showed that for any set $X$ of $n$
points in general position in the plane, there is a family $\F$ of at most
$20n\log\log n/\log n$ convex sets that separates $X$. On the other hand, there
is a set $X$ of $n$ points in the plane in general position for which any
separator of $X$, consisting of convex sets, has cardinality at least $n/(2\log
n+2)$.

It is natural to ask how these results about separating pairs of points could
be extended to separation of pairs of $k$-tuples of points.

\begin{defn}
Let $X$ be a set and $k$ a positive integer with $k\leq |X|-1$.
We say that a set $F\subset X$ \emph{containment-separates} a pair of
$k$-element subsets of $X$ if $F$ contains one of them and does not contain the
other. A family $\F$ of subsets of $X$ is a \emph{containment-separator} of
$k$-subsets of $X$ if for any two $k$-subsets $A$ and $B$ of $X$ there is at
least one member of $\F$ that separates them.
\end{defn}

Our goal is to select a small subfamily of $\F$ that containment-separates
the $k$-subsets of $X$.
As a generalization of the question discussed in \cite{GT13}, we denote by
$\ck{k}{n}{d}$ the minimum
number $c$ such that, for any set $X$ of $n$ points in general position in
$d$-space, there is a containment-separator of
$k$-subsets of $X$ which consists of $c$ convex sets.
This makes sense for
$k\le d+1$ and we have $\ck knd\le{\binom{n}{k}}$ because the convex hulls of
the $k$-subsets of $X$ containment-separate the $k$-subsets. However,
two $k$-subsets with the same convex hull are not containment-separated by any
convex set. Thus,  $\ck{k}{n}{d}$ does not exist if $d+2\leq k<n$.

The result quoted above from \cite{GT13} can be stated as
\begin{equation*}\label{eq:GT}
  n/(2\log n+2)\leq \ck{1}{n}{2}\leq 20n\log\log n/\log n.
\end{equation*}

Regarding higher dimensional point sets, it is observed in \cite{GT13} that for
any $n$ and $d$,
\begin{equation*}
c \frac{n}{\log^{d-1} n} \leq \ck{1}{n}{d} \leq C
\frac{n  \log \log n}{\log n}.
\end{equation*}

Our goal is to find bounds on $\ck{k}{n}{d}$ for $k\geq 2$.

\begin{thm}
\label{thm:ck}
For any $d=2,3,\ldots$, and any $k=2,\ldots,d+1$, there exists $c(k,d)>0$ such
that the following holds for any $n>k+1$.
\begin{eqnarray}
&c_2^1(n)&=2n-4
\label{line}\\
 \left\lfloor\frac{n-2}{2}\right\rfloor \leq& \ck{2}{n}{2} &\leq 2n-4,
\label{eq:geothmn4}\\
 \frac{n^2-2n-3}{8} \leq& \ck{3}{n}{2} &\leq n^2-n,
\label{eq:geothm5}\\
  c(k,d) \frac{n^{\lfloor(k+1)/2\rfloor}}{(\log n)^{\lfloor(2d-1-k)/2\rfloor}}
  \leq&
  \ck{k}{n}{d}&
  \leq 2{\binom{n}{k-1}}
\label{eq:kodd}
\end{eqnarray}
\end{thm}

It is a challenging problem to narrow the gap between the two bounds in
\eqref{eq:kodd}.
Theorems~\ref{thm:logP1} and \ref{thm:logPalpha} are shown in
Section~\ref{sec:logP}, Theorem~\ref{thm:satcond} in Section~\ref{sec:satcond},
Theorem~\ref{thm:VC} in Section~\ref{sec:VC}.
Theorem~\ref{thm:ck} is proved in Sections~\ref{sec:23dim}. 
We briefly discuss intersection-separation, a relative of
containment-separation, in Section~\ref{sec:intersectionseparation}.

\section{Proofs of Theorems~\ref{thm:logP1} and
\ref{thm:logPalpha}}\label{sec:logP}

\begin{proof}[Proof of Theorem~\ref{thm:logP1}]
We replaced our original proof of the first part of this theorem by a more 
elegant argument of Andr\'as M\'esz\'aros \cite{M14}, which was submitted as a 
solution to 
our problem at the Mikl\'os Schweitzer competition in 2014.

We regard $V=2^X$
as a vector field over $\mathbb F_2$ with respect to the symmetric difference 
of sets, which we denote by $+$. As $|X_i|\le n$ for all $i$ we can find a
family $U$ of at most $\lceil\log n\rceil$ subsets of $X$ (not 
necessarily in $\F$) that separates each $X_i$. Let $W$ denote the 
linear subspace of $V$ spanned by $U$. The translates of $W$ partition $V$, and 
thus, there is a translate $W+c$, more than half of whose members are in $\F$. 
Pick a $d\in (W+c) \cap \F$ and consider the set $\{x\in W:\; x+d\in\F\}$. 
This set has cardinality larger than $|W|/2$, so it spans $W$, and thus, it
contains a basis $Z$ of $W$. Now, $|Z|=\dim 
W\leq |U|\le\lceil\log n\rceil$. We claim that the set $S=\{d\}\cup
\{z+d:\;z\in Z\}$ separates each  $X_i$. To see this, observe that for any
pair of elements $x,y\in X$, the set of 
those elements of $V$ that do not separate $x$ and $y$ form a linear subspace 
of $V$, so if $S$ did not separate $x$ and $y$ neither did any set generated
by $S$. But this is not possible for $x,y\in X_i$ as $S$ generates each
element $z$ of $Z$ through
$z=d+(z+d)$ and thus $S$ also generates the subspace $W$ including the sets in
$U$, one of which separates $x$ from $y$. This finishes the proof of the first 
part of the theorem.

Now we prove the second part.
Let $X$ be the union of $k=2^{n-1}$, pairwise disjoint $n$-element sets, let
$\F$ consist of the elements of $2^X$ that intersect at least one class $X_i$
in $\emptyset$ or $X_i$.
Then $2^X \setminus \F$ has $\left( 2^n-2\right)^{2^{n-1}}$ elements, and
\[
\frac{|2^X \setminus \F|}{|2^X|} = \left( 1-\frac{1}{2^{n-1}} \right)^{2^{n-1}}
< \frac{1}{e} < \frac{1}{2}.
\]
Hence, $\F$ contains more than half of the elements of $X$.
On the other hand, assume that $\G \subset \F$ separates $X$, and let $G \in
\G$.
Let $X_i$ be a class in $X$ which is intersected by $G$ either in $\emptyset$
or in $X_i$.
Clearly, to separate $X_i$, we need at least $\lceil \log n \rceil$ more
elements of $\F$, which implies the assertion.
\end{proof}

Before presenting the details of the proof of Theorem~\ref{thm:logPalpha} we
introduce some notation and give a sketch of the proof.

In the setup of both Theorems~\ref{thm:logP1} and \ref{thm:logPalpha} we are
given a base set $X$ partitioned into the parts $X_i$ of size at most $n$ and
we want to refine this partition by selecting a small subset of a family
$\F\subseteq2^X$. In the proof of the latter theorem we will select this
separating subfamily in phases. After
selecting a subfamily $\F_1\subseteq\F$ we measure the progress by the maximal
size $m$ of a yet un-separated part of some $X_i$. Clearly, we need at least
$\log(n/m)$ sets to partition a set of size $n$ to sets of size $m$ or
smaller. Accordingly, we call the quantity $|\F_1|-\log(n/m)$ the
{\em loss} incurred in decreasing the part size from $n$ to $m$.

In this
terminology we can phrase Theorem~\ref{thm:logP1} as stating that if the
density of $\F$ is above $1/2$, then we can decrease the part size all the way 
to
$1$ with a loss of less than $2$. Similarly, Theorem~\ref{thm:logPalpha}
states that if $\F$ has density $\alpha$ and separates all $X_i$, then we can
decrease the part size to $1$ with a loss of
$O(\log(1/\alpha)\log\log(1/\alpha))$. Note that a loss of $\log(1/\alpha)$ is
unavoidable in certain cases, for example if we have a small set $Y\subset
X_1$ and $\F$ consists of all sets disjoint from $Y$, containing $Y$, or
having size $1$. We are not sure if the $\log\log(1/\alpha)$ factor is needed.

We will prove Theorem~\ref{thm:logPalpha} by constructing the separating
family in phases. The first stage will reduce the part size to at most
$1/\alpha$ for a loss of less than $2$. This stage is a generalization of
Theorem~\ref{thm:logP1} and proved very similarly.

In the second phase we further reduce the
part size to $O(\log(1/\alpha))$ for a loss of $O(\log(1/\alpha))$. We will
select the separating sets in this phase one by one, but we remark that
selecting them at once one can decrease the loss incurred in this phase to a
constant for the small price of reducing the part size to $O(\log^2(1/\alpha))$
or even $O(\log(1/\alpha)\log\log(1/\alpha))$ (instead of $O(\log(1/\alpha))$
as presented here). Unfortunately, the loss
incurred in the third phase is much larger and dominates the losses in the
other phases.

In the third phase we make sure that all but $O(\log(1/\alpha))$ elements of
$X$ form singleton parts, while in the final fourth phase we separate the
remaining few non-singleton parts to ``atomic'' parts not even separated by
$\F$.

Let $A$ and $Y$ be finite sets. We say that $A$ {\em cuts $Y$ well} if
$|Y|/4\le|Y\cap A|\le3|Y|/4$. In the second phase the following trivial
observation is going to be useful:

\begin{lemma}\label{well}
Let $Y$ be a finite set of size $m\ge2$. The density $\alpha$ of
the subsets of $Y$ that do not cut $Y$ well satisfies $\alpha\le1/2$ and
$\alpha\le2^{-m/10}$.
\end{lemma}

In the third phase we use the following result of Brace and Daykin \cite{BD71}.

\begin{thm}[Brace and Daykin, 1971]\label{thmbd}
Let $t>1$ be an integer, $Y$ be a set of size $s$ and $\F'$ be a subset of
$2^Y$ of density exceeding $(t+2)/2^{t+1}$. If $\bigcup\F'=Y$, then there are 
$t$
elements of $\F'$ whose union is also $Y$.
\end{thm}

\begin{proof}[Proof of Theorem~\ref{thm:logPalpha}]
In the first phase of the selection of the separating subfamily of $\F$ we
decrease the maximal size of a part from $n$ to at most $1/\alpha$. In case
$n\le1/\alpha$ one can simply skip this phase.

We mimic the proof of Theorem~\ref{thm:logP1}. We assume without loss of
generality, that $n=\max_i|X_i|$ and set $r=\lceil\log n\rceil$. We let
$V=\mathbb F_2^r$ be the $r$-dimensional vector space over the two element
field with the usual inner product and choose $f:X\to V$ that is injective on
each $X_i$. We set $W=\{O_x\mid x\in V\}$, where $O_x$ consists of the
elements $a$ of $X$ with $f(a)$ not orthogonal to $x$. We regard $2^X$ as a
group with the symmetric difference operation (denoted by $+$) and note that
$W$ is a subgroup. As the density of $\F$ is $\alpha$ we can choose $C\in\F$
such that the density of $\F$ within the coset $C+W$ is at least $\alpha$,
that is the set $S=\{ x \in V: C+O_x \in \F \}$ satisfies $|S|\ge\alpha2^r$.

With $\alpha\le1/2$ the set $S$ does not necessarily generate the whole of
$V$. But we can still choose a basis $B\subseteq S$ for the subspace of $V$
generated by $S$ and it is easy to see that the set $\F_1=\{C+O_x\mid x\in
B\}\cup\{C\}$ of size $|B|+1$ separates each $X_i$ into parts of size at most
$2^{r-|B|}\le1/\alpha$. Thus, we have decreased the part size to at most
$1/\alpha$ for a loss of less than $2$. We have $|\F_1|<\log n+2$.

In the second phase we do similarly as in the first phase but we do not have a
linear structure on $V$ any more. Let $m$ be the maximal size of a part in the
current partition of $\F$. We take $V$ to be an $m$ element set and select a
function
$f:X\to V$ that is injective in every part. We set $W=\{f^{-1}(H)\mid
H\subseteq V\}$. As before, $W$ is a subgroup of $2^X$ (considered with the
symmetric difference). We select $C\in\F$ such that the density of $\F$ in the
coset $C+W$ is at least $\alpha$, that is, the set $S=\{H\subseteq V\mid
C+f^{-1}(H)\in\F\}$ satisfies $|S|\ge\alpha2^m$. In case $S$ contains a set
$H$ that cuts $V$ well, we choose one such set and include the sets $C$ and
$C+f^{-1}(H)$ in our separating family. This makes all the parts in the current
partition be at most $3m/4$.

By Lemma~\ref{well} the selection of the set $H\in S$ cutting $V$ well is
possible as long as $\alpha>2^{-m/10}$, that is $m>10\log(1/\alpha)$. In the
second phase we repeat the above procedure till all parts of the current
partition is of size at most $10\log(1/\alpha)$ and call $\F_2$ the set of
separating sets collected. As the maximal size of a part was at most
$1/\alpha$ in the beginning of phase two, and this maximal size decreases by a
constant factor in each round when we add two sets to $\F_2$ we have
$|\F_2|=O(\log(1/\alpha))$.

In the third phase we use a different strategy. Let $\{X'_i\mid i\in I\}$ be
the set of parts in the current partition of $X$. We call a part $X'_i$
{\em good} if there is a set $A\in\F$ cutting
$X_i$ well and we call $X'_i$ {\em bad} if it has at least $2$ elements, but
it is not good. Let $G=\{i\in I\mid X'_i\mbox{ is good}\}$. For $i\in G$ we
select a
set $B_i\subset2^{X'_i}$ of size $|B_i|=2^{|X'_i|-1}$ such that all sets in
$B_i$ cuts $X'_i$ well and we have a set $A\in\F$ with $A\cap X'_i\in
B_i$. This is possible as by Lemma~\ref{well} at least half the subsets of
$X'_i$ cut $X'_i$ well.

We define the function $f:2^X\to2^G$ by setting $f(A)=\{i\in G\mid A\cap
X'_i\in B_i\}$. Let $\F'=f(\F)$. Clearly, $f$ takes all values an equal number
of times, thus the density of $\F'$ in $2^G$ is at least the density $\alpha$
of $\F$ in $2^X$. Note that for all $i\in G$ we have $A\in\F$ with $i\in f(A)$,
thus we have $\bigcup\F'=G$. We choose $t=O(\log(1/\alpha))$ such that
$(t+2)/2^{t+1}<\alpha$ and apply Theorem~\ref{thmbd} to find $t$ sets
$I_1,\ldots,I_t\in\F'$ with $\bigcup_{i=1}^tI_i=G$. We find sets $A_i\in\F$
with $f(A_i)=I_i$ and include these $t$ sets in our partitioning family. Note
that if $x\in X$ is contained in a good part $X'_i$ of size $m$, then the
size of the part containing $x$ after considering these $t$ new separating
sets is at most $3m/4$.

We repeat the above procedure for
$\lceil\log(10\log(1/\alpha))/\log(4/3)\rceil=O(\log\log(1/\alpha))$ times and
obtain $\F_3$ as the union of all the elements of $\F$ we selected. We clearly
have $|\F_3|=O(\log(1/\alpha)\log\log(1/\alpha))$.

We call $x\in X$ fully separated if it forms a singleton part in the current
partition after the third phase. Clearly, if $x$ is not fully separated, it
must have been in a bad part at some time. Let us choose the earliest bad part
containing $x$ and consider all the distinct sets $Y_1,\ldots,Y_j$
obtained this way from not fully separated elements. Clearly, these sets are
pairwise disjoint. Let $Y=\bigcup_{i=1}^jY_i$. By Lemma~\ref{well} the ratio of
subsets of $X$ not cutting $Y_i$ well is at most $2^{-|Y_i|/10}$. A random
subset $A\subseteq X$ intersects the sets $Y_i$ independently, thus the
probability that it cuts none of the $Y_i$ well is at most
$\prod_{i=1}^j2^{-|Y_i|/10}=2^{-|Y|/10}$. By the definition of bad parts, no
set $A\in\F$ cuts any of the sets $Y_i$ well, so we have
$\alpha\le2^{-|Y|/10}$. Therefore, all non-singleton parts in the current
partition after the third phase is contained in the set $Y$ of size at most
$10\log(1/\alpha)$.

Finally in the last phase we select any set in $\F$ that refines our current
partition and repeat this process until no further refinement is
possible. Clearly, the set $\F_4$ selected in this phase satisfies
$|\F_4|<|Y|=O(\log(1/\alpha))$.

The family $\F_1\cup\F_2\cup\F_3\cup\F_4$ separates any pair of elements in
the same part $X_i$ that is separated by $\F$ and the size of this set is
$\log n+O(\log(1/\alpha)\log\log(1/\alpha))$. This finishes the proof of
Theorem~\ref{thm:logPalpha}.
\end{proof}

\section{Proof of Theorem~\ref{thm:satcond}}\label{sec:satcond}
For
$0<\varepsilon<1$, we call a constraint \emph{$\varepsilon$-good} if at least
$\varepsilon|\F|$ members of $\F$ satisfy it, and \emph{$\varepsilon$-bad} 
otherwise.

The proof of the first statement of Theorem~\ref{thm:satcond} consists of two
steps.
First, a standard application of the probabilistic method shows that for any
$\varepsilon>0$, all $\varepsilon$-good constraints can be satisfied by $(\log
N)/\varepsilon$ randomly chosen members of $\F$.
Second, we show that if $\varepsilon$ is set sufficiently small as a function of
$m$ and $\alpha$ but independent of $|X|$ and $N$, then the number of
$\varepsilon$-bad constraints is bounded by another value $Z$ depending on $m$ 
and
$\alpha$, and independent of $|X|$ and $N$.
Since all constraints are satisfiable, it means that adding $Z$ (well
chosen) members of $\F$ to the $(\log N)/\varepsilon$ random members satisfying
the $\varepsilon$-good constraints we obtain a collection satisfying all the $N$
constraints.

To prove the first step, let $C$ be a $\varepsilon$-good constraint, and choose
$(\log N)/c$ members of $\F$ randomly, uniformly.
\[
\Pe(C \mbox{ is not satisfied by any of the chosen sets})\leq
(1-c)^{\frac{\log N}{c}}
<
\frac{1}{N}.
\]
Thus, with non-zero probability, all the at most $N$ $\varepsilon$-good
constraints are satisfied by the randomly chosen members of $\F$.

For the second step let $a$ be a large enough number depending on $\alpha$ and
$m$ to be set shortly. Assume that the supports of $a2^m$ distinct
$\varepsilon$-bad constraints form a ``sunflower'', that is any two of them
intersect in the same set $C$. Let $b=|C|$ and note that $0\le b\le
m-1$. Clearly, we can find $a$ of these constraints, say $(V_i,W_i)$ for $1\le
i\le a$ that contain the elements of $C$ ``on the same side'', that is $C$ is
the support of a constraint $(V_0,W_0)$ with $V_0\subseteq V_i$ and
$W_0\subseteq W_i$ for $i=1,\ldots,a$. Now consider a uniform random subset
$A\subseteq X$. This set satisfies the constraint $(V_0,W_0)$ with probability
$2^{-b}$. If it satisfies $(V_0,W_0)$, then the conditional probability that
it also satisfies $(V_i,W_i)$ is $2^{b-m}$ for any $i\ge1$. From the sunflower
property we see that assuming $A$ satisfies $(V_0,W_0)$ the events that $A$
satisfy $(V_i,W_i)$ are mutually independent, so the overall probability $P$
that $A$ satisfies at least one of them is exactly
$$P=2^{-b}(1-(1-2^{b-m})^a).$$
On the other hand, $A$ has the chance at least $\alpha$ to be in $\F$, and
assuming it is in $\F$, it has a chance of less than $\varepsilon$ to satisfy 
any
of those constraints. Therefore we have
$$P<(1-\alpha)+a\alpha\varepsilon.$$
Let us select $\varepsilon>0$ small enough and $a$ large enough (depending on 
$m$
and $\alpha$) such that
$$2^{-b}(1-(1-2^{b-m})^a)\ge(1-\alpha)+a\alpha\varepsilon$$
holds for any $0\le b\le m-1$. This is possible as increasing $a$ the left
hand side approaches $2^{-b}\ge2^{1-m}$ and for fixed $a$, and $\varepsilon$
approaching $0$, the right hand side approaches $1-\alpha<2^{1-m}$.

With this choice of $\varepsilon$ and $a$ the imminent contradiction in the last
three displayed equations shows that the supports of no $a2^m$ $\varepsilon$-bad
constraints form a sunflower. By the sunflower lemma of Erd\H os and Rado
\cite{ER60} we find that the total number of $\varepsilon$-bad constraints is at
most $Z=m!2^m(a2^m)^m$. (The extra $2^m$ factor is coming from the possibility
that many $\varepsilon$-bad constraints may have the same support.) This bound
is independent of $N$ and $|X|$ as claimed and finishes the proof of the first
statement of Theorem~\ref{thm:satcond}.

For the second part, for any $m$ and $N$ we construct a base set $X$, a family
$\F$ of density strictly above $1-2^{1-m}$ and $N$ one
sided constraints $(V_i,\emptyset)$ of size $m$, all satisfied by exactly one
member of the family $\F$ and such that different constraints are satisfied by
different members of $\F$. Thus, all constraints are satisfied by $\F$ but no
subset of cardinality less than $N$ satisfies them all.

For this we set $|X|=N+m-1$ and identify a subset $Y\subset X$ of size
$|Y|=m-1$. Let $\F$ consist of all the subsets of $X$ not containing $Y$ plus
all the $m$ element subsets. We select the constraints $(V_i,\emptyset)$ with
all possible $m$-subsets $V_i$ of $X$ containing $Y$.

\section{Proof of Theorem~\ref{thm:VC}}\label{sec:VC}

Assume without loss of generality that $U$ is the set of positive integers.
Since the VC-dimension of $\mathcal{F}$ is $d$, for any $(d+1)$-element subset
$A=\{a_1, a_2, \ldots, a_{d+1}\}\subset U$ with $a_1<a_2<\ldots<a_{d+1}$, there
is a set $T(A)\subseteq\{1,2,\ldots,d+1\}$ such that no $F\in \mathcal{F}$
satisfies $F\cap A=\{a_i | i\in T(A)\}$. If there is more than one such set, we
fix $T(A)$ arbitrarily, and we call it the \emph{type} of $A$.

By Ramsey's Theorem, there is a set $X=\{y_1,y_2,\dots\, y_n\}\subset U$ with
$y_1<\ldots<y_n$, such that all $(d+1)$-element subsets of $X$ have the same
type $T$.

Let us call the index $1\le i\le d$ {\em regular} if $T$ separates
$i$ from $i+1$ and {\em singular} otherwise. We let $k$ stand for the number
of singular indices.

Consider any $F\in \mathcal{F}$. We claim that there is a set $F^*$ of at most
$k$ elements of $X$ and a partition of $X$
into at most $d-k+1$ intervals such that the symmetric difference $F+F^*$
does not separate any two elements in the same interval. To see this
consider the greedy process looking for indices
$i_1<i_2<\ldots<i_{d+1}$ such that for the set
$H=\{y_{i_1},\ldots,y_{i_{d+1}}\}$ we have $H\cap F=\{y_{i_j}\mid j\in
T\}$. As such a set could not have type $T$, we cannot find all these indices,
nevertheless, some of the indices can be found by a greedy process (for
example, if $1\in T$, then we start with $i_1=\min\{j\mid y_j\in
F\}$). We can satisfy the claim by making $F^*$ consist of the elements
$y_{i_j}$, where $j$ is singular and starting a new interval in the partition
of $X$ at every element $y_{i_j}$ with $j$ regular.

Now assume the family $\{F_1,\ldots,F_m\}\subseteq\F$ separates $X$. Let
$F_i^*$ be the corresponding sets of size at most $k$ and consider the set
$V=X\setminus\bigcup_{i=1}^mF_i^*$. We have $|V|\ge n-mk$ determining at
least $n-mk-1$ neighboring pairs of elements. By the above
claim each set $F_i$ separates at most $d-k$ of these neighboring elements,
thus we must have $n-mk-1\le m(d-k)$, that is, $m\ge(n-1)/k$. This completes
the
proof
of Theorem~\ref{thm:VC}.

\section{Containment-separation}\label{sec:23dim}

In this section we prove Theorem~\ref{thm:ck}.

We start with the proof of (\ref{line}). Let $H=\{x_1,\ldots,x_n\}$ be an
arbitrary $n$ element subset of the real line and assume
$x_1<\ldots<x_n$. Note that the family
$\{[x_1,x_i]\mid1<i<n\}\cup\{[x_i,x_n]\mid1<i<n\}$ of $2n-4$ convex sets
containment-separates all the pairs in $H$. This proves $c^1_2(n)\le2n-4$. On
the other hand realize that no convex set containment-separates more than one
of the $2n-4$ pairs $(\{x_1,x_i\},\{x_1,x_{i+1}\})$ for $1<i<n$ and
$(\{x_i,x_n\},\{x_{i-1},x_n\})$ for $1<i<n$. This proves $c^1_2(n)\ge2n-4$.

We will use the following monotonicity property:
\begin{equation*}
  \ck{k}{n}{d}\leq \ck{k}{n}{d-1},
\end{equation*}
  for all choices of $k\le d$ and $n$. To see this consider a set $X$ of $n$
  points in general position in
  $d$-space and find a projection $\pi$ to $(d-1)$-space such that $\pi(X)$ is
  again $n$ points in general position. A generic projection $\pi$ satisfies
  this. Now find $\ck kn{d-1}$ convex sets to containment
  separate all the $k$-sets of $\pi(X)$ and consider the inverse images of
  these sets for the projection $\pi$. Clearly, these sets are convex and they
  containment-separate the $k$-subsets of $X$.

A similar monotonicity also holds in $k$ if $n>k+1$:
$$\ck knd\le\ck{k+1}nd.$$
This is because any collection of sets containment-separating the
$(k+1)$-subsets of an $n$-set $H$ also containment-separate the
$k$-subsets. Indeed, if $A$ and $B$ are $k$-subsets, then any set containment
separating $A\cup\{x\}$ from $B\cup\{x\}$ also containment-separates $A$ from
$B$. This trick works if we can choose $x\in H$ outside $A\cup B$. In case
$A\cup B=H$ we pick $x\in A\setminus B$ and $y\in B\setminus A$ and use that
any set containment-separating $A\cup\{y\}$ from $B\cup\{x\}$ also containment
separates $A$ from $B$. This latter trick fails if $|A\cap B|=k-1$, but then
$|A\cup B|=k+1<n$.

Next we prove \eqref{eq:geothmn4}.
The upper bound follows from \eqref{line} via the monotonicity mentioned
above.

For the lower bound in \eqref{eq:geothmn4} let $p=(1,1)$ and consider the set
$X$ consisting of $p$ and $n-1$ points on the unit circle around the origin,
all in the first quadrant. This set is in general position. Let us denote the
points of $X$ on the circle $x_1,\ldots,x_{n-1}$ in order of increasing
$x$-coordinate. The property of the arrangement we use is that the the convex
hull of $\{x_i,x_j,p\}$ contains all $x_k$ with $i<k<j$. This implies that any
convex set can containment-separate at most two of the pairs of
two element sets $(\{x_i,p\},\{x_{i+1},p\})$ for $1\le i\le n-1$ and the
stated lower bound on $c_2^2(n)$ follows.

Next we turn to the upper bound in \eqref{eq:kodd}. By the monotonicity
mentioned above it is enough for us to prove $\ck kn{k-1}\le2{\binom{n}{k-1}}$. 
Let us consider a set $X$ of $n$ points in general position in the
$(k-1)$-dimensional  space. Each $(k-1)$-subset of $X$ determines a
hyperplane. Consider all the closed half-spaces bounded by one of
these hyperplanes. This is a collection of $2{\binom{n}{k-1}}$ convex sets and
it containment-separates
all $k$-subsets of $X$. Indeed, if $A$ and $B$ are distinct $k$-subsets of $X$,
then either the convex hull of $A$ does not contain $B$, or vice versa. In the
former case a supporting half-space of the convex hull of $A$
containment-separates the sets.

We turn to the proof of \eqref{eq:geothm5}.
The upper bound is a special case of the upper bound in \eqref{eq:kodd}. For
the lower bound we give a construction.

Let us assume $n\ge8$. Let $X_{\mathrm{ex}}$ be the vertex set of a regular
$k$-gon around the origin for $k=2\lfloor n/4\rfloor$. We call the opposite
pairs of points in $X_{\mathrm{ex}}$ a {\em diameter}. Let us find a
point $x$ distinct from the origin but so close to it that it is contained in
the interior of the convex hull of any two diameters. We further assume that
$X_{\mathrm{ex}}\cup\{x\}$ is in general position. Let $l=n-k$ and
$X_{\mathrm{in}}=\{x_i\mid 1\le i\le l\}$, where $x_i=\frac ilx$.

Consider the pairs of 3 element sets $\{\{p,-p,x_i\},\{p,-p,x_{i+1}\}\}$,
where $p \in X_{\mathrm{ex}}$ (so $\{p,-p\}$ is a diagonal) and $1\le i\le
l-1$. No set containing no diagonal can containment-separate any of
these pairs, and convex sets containing more than one diagonal contain all
$x_i$ in their interior, so they do not separate these pairs either. A convex
set $S$ containing a single diagonal $\{p,-p\}$ separates at most one of these
pairs, since, if $x_i\in S$, then we also have $x_j\in S$ for all $j<i$.

This shows that $X_{\mathrm{ex}}\cup X_{\mathrm{in}}$ is a good choice for a
hard to separate set, but it is not in general position. Fortunately, the above
arguments are robust against small perturbations. Let us obtain $X$ as the
union of $X_{\mathrm{ex}}$ and a set $\{x'_i\mid 1\le i\le n\}$ where $x'_i$
is $\varepsilon$-close to $x_i$ but perturbed in such a way that $X$ is in
general position.

It is easy to see that if $\varepsilon>0$ is small enough, then
each convex set can containment-separate at most one of the pairs
$(\{p,-p,x'_i\},\{p,-p,x'_{i+1}\}$. To containment-separate all these pairs,
one needs at least $(l-1)k/2$ convex sets, proving the lower bound in 
\eqref{eq:geothm5}.

Finally, we prove the lower bound in \eqref{eq:kodd}. For this we need the
following result.
For any $n>d\ge 2$, K\'arolyi and Valtr
\cite{KV03} constructed a set of $n$ points in the $d$-dimensional space
which contains at most $c_d\log^{d-1}n$ points in convex position, where $c_d$
depends only on $d$. We call such a point set a {\em K\'arolyi-Valtr
construction}.

Let us assume that $k$ is odd as the case of even $k$ comes from the
monotonicity of $c_k^d(n)$ in $k$. We need to construct a set $X$ of $n$
points in general position in $d$-space, whose $k$-subsets is hard to
containment-separate with convex sets. Let $m=\frac{2d+1-k}{2}$.
Take a set $A$ of size $\frac n4$ and a set $A'$ of size $m$ such that all
the $\frac n4+m$ points in $A\cup A'$ are at unit distance from the origin and
the unit vectors corresponding to any $d$ of them are linearly
independent. Let $X_{\mathrm{ex}}=-A\cup A$, where $-A=\{-p\mid p\in A\}$. We
call the points $p$ and $-p$ in $X_{\mathrm{ex}}$ an {\em opposite pair}.
Let $F$ be an $m$-flat passing through the origin and the points of $A'$.
Consider a ball $B$ centered at the origin that is so small that the convex 
hull of any
$d$ opposite pairs of  $X_{\mathrm{ex}}$ contains $B$
in its interior. Let $X_{\mathrm{in}}$ be
an $m$-dimensional K\'arolyi-Valtr construction of $\frac{n}{4}$
points in $F\cap B$ and assume (without loss of generality) that
$X_{\mathrm{in}}$ as an $m$-dimensional set is in general position. (Note,
however, that neither $X_{\mathrm{ex}}$ nor $X_{\mathrm{in}}$ is in general
position in $d$-space.)

We claim that if a convex set $S$ contains $\frac{k-1}2$ opposite
pairs, then the points of $X_{\mathrm{in}}$ on the boundary of $S$ are in convex
position. Indeed, otherwise there would be a set $H$ of $m+1$ points
in $X_{\mathrm{in}}$, all on the boundary of $S$ and forming a simplex in $F$
such that the simplex contains yet another point $x\in X_{\mathrm{in}}$ on the
boundary of $S$. The contradiction comes from the fact that in this case the
convex hull of the union of $H$ and the opposite pairs in $S$ has $x$ in its
interior.

As a consequence, we see that if a convex set contains $\frac{k-1}2$ opposite
pairs, then its boundary contains at most $c_m\log^{m-1}n$ points of
$X_{\mathrm{in}}$.

Set $\varepsilon>0$ very small and let $X$ be a point set consisting of a
point $\varepsilon$-close to each point in $X_{\mathrm{ex}}$ and two distinct
points $\varepsilon$-close to each point in $X_{\mathrm{in}}$. We choose the
points in $X$ to be in general position and let $f:X_{\mathrm{ex}}\to X$,
$f_1:X_{\mathrm{in}}\to X$ and $f_2:X_{\mathrm{in}}\to X$ be the functions
showing our choices.

Consider a set $H$ of
$\frac{k-1}2$ opposite pairs, and a point $x\in X_{\mathrm{in}}$ and let
$P(H,x)$ be the pair $(f(H)\cup\{f_1(x)\},f(H)\cup\{f_2(x)\})$. We claim that if
$\varepsilon$ is small enough, then any convex set can containment-separate at
most $c_m\log^{m-1}n$ pairs $P(H,x)$ with a fix $H$. Indeed, to separate the
pair $P(H,x)$ the convex set must contain $f(H)$ and has to have a boundary
point on the (short) interval $f_1(x)f_2(x)$. In the limit for $\varepsilon\to0$
we find a convex set containing $H$ and having $x$ on its boundary. As we saw
above this is possible for at most $c_m\log^{m-1}n$ points $x\in
X_{\mathrm{in}}$.

If a convex set $S$ contains $d$ opposite pairs, then it contains
$X_{\mathrm{in}}$ in its interior. Therefore, for small enough $\varepsilon$, a
convex set containing $f(H)$ for a collection $H$ of $d$ opposite pairs
containment-separates no pair $P(H,x)$.

To containment-separate $P(H,x)$ a (convex) set must contain $f(H)$, so the
above bounds mean that (again, for small enough $\varepsilon$) no convex set
containment-separates more than $c_m{\binom{d-1}{\frac{k-1}2}}\log^{m-1}n$ of 
the
pairs $P(H,x)$. Comparing this with the total number of $\frac n4\binom{\frac
n4}{\frac{k-1}2}$ of the pairs $P(H,x)$ shows that we need many convex
sets to containment-separate the $k$-subsets of $X$. This finishes the proof of
the lower bound in \eqref{eq:kodd} and with that the proof of
Theorem~\ref{thm:ck}.

\section{A remark: Intersection-separation}\label{sec:intersectionseparation}

In the geometric setting, we discussed containment-separation. We can
extend the notion of separation of points to $k$-tuples in another way as well.
We say that a set $F$ \emph{intersection-separates} a pair of
$k$-element subsets of $X$, if $F$ intersects one of them and is disjoint from
the other. A family $\F$ of subsets of $X$ \emph{intersection-separates} the
$k$-element subsets of $X$ if, for any pair of $k$-element subsets of $X$,
there is a member of $\F$
that intersection-separates that pair. And thus, we can define the
intersection-separation numbers as

\begin{equation*}
\ik{k}{n}{d}=
\max_{
 \begin{array}{c}
{\scriptstyle X\subset\Red, |X|=n} \\
{\scriptstyle X \mbox{ \tiny in general pos.}}
 \end{array}
}
\min\left\{|\G| : \G\subset\C^d
\begin{array}{c}
 \mbox{ an intersection-separator of}\\
 k\mbox{-subsets of } X
\end{array}
\right\},
\end{equation*}
where $\C^d$ denotes the family of convex subsets of $\Red$.

The number $i_k^d(n)$ is always defined and at most $n-1$ as the singleton
subsets intersection-separate, even if we omit one of them.

The following monotonicity properties can be verified exactly as for
containment-separation.

\begin{equation}\label{eq:monotoneik}
 \ik{k}{n}{d}\leq \ik{k}{n}{d-1},
\;\mbox{ and }\;\;
 \ik{k}{n}{d}\leq \ik{k+1}{n}{d}
\end{equation}

\begin{prop}\label{prop:ik}
For any $d=2,3,\ldots$ there is a constant $c_d>0$ such that for any $n>k\ge2$
we have
\begin{eqnarray}
 \left\lfloor\frac{n+3}{6}\right\rfloor \leq& \ik{k}{n}{2} & \leq n-1
\label{eq:ik1}\\
 c_d \frac{n}{\log^{d-1} n} \leq&\ik{k}{n}{d} &\leq n-1
\label{eq:ik2}
\end{eqnarray}
\end{prop}


We have already mentioned the upper bounds. The lower bound in
\eqref{eq:ik2} can be proved by replacing each point of a K\'arolyi--Valtr
construction (see Section~\ref{sec:23dim}) by a pair of twins (two very close
points). We provide a construction to show the lower bound in \eqref{eq:ik1}.

By \eqref{eq:monotoneik} it is sufficient to consider the case $k=2$.
Suppose without loss of generality that $n=3m$.
We give the points in polar coordinates $(r, \phi)$.
Let $\varepsilon>0$ be very small.
For $0\le i\le m-1$, let $p_i$ be the point $(1-\varepsilon, 2i\pi/m)$,
$q_i$ the point $(1, 2i\pi/m+\varepsilon^2)$, and
$r_i$ the point $(1, 2i\pi/m-\varepsilon^2)$.
Let $X$ be the set of these $n=3m$ points. For any $i$, $0\le i\le m-1$,
consider the following pair of pairs:
$\left(\{p_i,q_i\},\{p_i,r_i\}\right)$.

To finish the proof, we claim that no convex set intersection-separates more
than two of these $m$ pairs of pairs.
Indeed, suppose that a convex set $K$ intersection-separates the pairs
corresponding to the indices $i,j$ and $k$. We may assume
that the greatest angle of the triangle $p_ip_jp_k$ is at $p_j$. Now, $K$ 
contains one of $q_i$ and $r_i$, one of $q_j$ and $r_j$, and one of $q_k$ and 
$r_k$. It is easy to
see that, if $\varepsilon$ is small enough, $K$ contains $p_j$, a contradiction.


\end{document}